\documentclass[preprint,12pt]{elsarticle}
\journal{xxxxxxxx}
\usepackage{anyfontsize,t1enc}
\usepackage{lscape}
\usepackage{array}
\usepackage{booktabs}

\usepackage[noend]{algpseudocode}
\usepackage{here}
\let\oldReturn\Return
\renewcommand{\Return}{\State\oldReturn}

\usepackage{lineno} % to enumerate the lines

\usepackage{babelbib,multicol}

\usepackage[caption=false,font=normalsize,labelfon
t=sf,textfont=sf]{subfig}

\usepackage[english]{babel}

\usepackage{amsmath,dsfont,amsthm,cancel,amsfonts,amssymb,mathrsfs,amsthm,cmll,nicefrac}
\usepackage{textcomp}

\usepackage{algorithm}

\usepackage{color}

\usepackage{enumitem}
\usepackage{graphics}
\usepackage{lineno}

\usepackage{hyperref}
\usepackage{MnSymbol,stmaryrd} %,yhmath} 
\usepackage{manfnt}

\usepackage{tikz}
\usetikzlibrary{automata} 

\definecolor{rojo}{RGB}{172,6,6}
\definecolor{negro}{RGB}{25,6,6}

\newcolumntype{M}[1]{>{\hbox to #1\bgroup\hss$}l<{$\egroup}}

\makeatletter
\newcommand\@brcolwidth{2.9em}
\newenvironment{brmatrix}{%
    \left(%
    \hskip-\arraycolsep
    \new@ifnextchar[\@brarray{\@brarray[\@brcolwidth]}%
}{%
    \endarray
    \hskip -\arraycolsep
    \right)%
}
\def\@brarray[#1]{\array{r*\c@MaxMatrixCols {M{#1}}}}
\makeatother
\DeclareMathOperator{\Maxx}{Max}
\DeclareMathOperator{\Minn}{Min}
\DeclareMathOperator{\supp}{supp}

\DeclareMathOperator{\TFN}{TFN}
\DeclareMathOperator{\OT}{OT}
\providecommand{\leqot}[0]{\leq_{\OT}}
\providecommand{\geqot}[0]{\geq_{\OT}}
\newcommand{\ITEM}[1]{\noindent{#1}}
\providecommand{\Vector}[1]{\overrightarrow{#1}}
\providecommand{\NFT}[1]{(a_{#1},b_{#1},c_{#1})}

\providecommand{\crispy}[1]{\widetilde{#1}}

\providecommand{\cortea}[1]{\nicefrac{\Large{#1}}{\alpha}}

\providecommand{\operador}[3]{#1:#2\longrightarrow #3}
\providecommand{\con}[1]{\mathds{#1}}

\providecommand{\Sup}[1]{\underset{#1}{\sup~}}
\providecommand{\NFR}[0]{\mathcal{F}(\mathds{R})}

\providecommand{\NFRN}[1]{\mathcal{F}_{-}(\mathds{R})}

\providecommand{\TFNR}{\mathfrak{T}(\con{R})}

\providecommand{\suma}[2]{\overset{#1}{\underset{#2}{\sum}}}

\newcommand{\prt}[1]{\langle #1\rangle}

\newtheorem{teorema}{Theorem}[section]
\newtheorem*{tricotomia}{$OT$-Trichotomy law for $\TFNR$}
\newtheorem{proposicion}{Proposition}[section]
\newtheorem{corolario}{Corollary}[section]

\newtheorem{definicion}{Definition}[section]
\newtheorem{observacion}{Remark}[section]

\newtheorem{ejemplo}{Example}[section]

\begin{document}

\begin{frontmatter}

\title{Averaging functions on triangular fuzzy numbers}

\author{Roberto D\'iaz$^a$, Aldryn Aparcana$^{b}$, Diego Soto$^c$, Nicolás Zumelzu$^c$, Jos\'e Canum\'an$^{d}$, Alvaro Mella$^{c}$, Edmundo Mansilla$^{c}$ and Benjam\'in Bedregal$^e$}

\cortext[]{Corresponding author: nicolas.zumelzu@umag.cl }

\ead{roberto.diaz@ulagos.cl, aldryn.aparcana@unica.edu.pe,disoto@umag.cl, nicolas.zumelzu@umag.cl, jose.canuman@umag.cl, alvaro.mella@umag.cl, edmundo.mansilla@umag.cl and bedregal@dimap.ufrn.br}

\address[a]{Departamento de Ciencias Exactas, Universidad de Los Lagos, Osorno, Chile}

\address[b]{Facultad de Ciencias, Universidad Nacional San Luis Gonzaga, Ica, Peru}

\address[c]{Departamento de Matem\'atica y F\'isica, Universidad de Magallanes, Punta Arenas, Chile}

\address[d]{Departamento de Ingenier\'ia en Computaci\'on, Universidad de Magallanes, Punta Arenas, Chile}
\address[e]{Departamento de Inform\'{a}tica e Matem\'{a}tica Aplicada, Universidade Federal do Rio Grande do Norte, Natal, Brazil}

\begin{abstract} Admissible orders on fuzzy numbers are total orders which refine a basic and well-known partial order on fuzzy numbers. In this work, we define an admissible order on triangular fuzzy numbers (i.e. $\TFN$'s) and study some fundamental properties with its arithmetic and their relation with this admissible order. In addition, we also introduce the concepts of average function on $\TFN$'s and studies a generalized structure of the vector spaces.  In particular we consider the case of $\TFN$'s.


\end{abstract}

\end{frontmatter}

\section{Introduction}

Zadeh in \cite{zadeh1965fuzzy} proposed the concept of fuzzy sets to formalize and model the ambiguities and inaccuracies of language such as temperature, age, and speed. A class of fuzzy sets, called of fuzzy numbers, were introduced in \cite{zadeh1965fuzzy} to model a quantity that is imprecise, rather than exact, as it is in all traditional mathematics. This structure has been studied in countless works, and has been dedicated to: arithmetic operations  and other related notions for fuzzy numbers and their properties (see \cite{dubois1978h,hanss2005applied,klir1995}). Additionally, there are several proposal of partial orders (totals and non totals) for fuzzy numbers or a subclass of the fuzzy numbers as for example in \cite{TAsmus2017,kosinski2017,Roldan-Lopez-de-Hierro18,wang2014total}. More recently, Zumelzu et al. in \cite{zumelzu2022}, in the light of the notion of admissible orders in the context of extensions of fuzzy set theory as in \cite{bustince2013generation,Annax21,Laura16b,matzenauer2021strategies},  proposed the notion of admissible orders for fuzzy numbers as total orders which refines the order of Klir and Yuan in \cite{klir1995}.

To talk about vector spaces, we have to highlight the hard work of mathematicians who allowed the definition and axiomatization of this concept. Among them, we can find \cite{hamilton1853wr}, that by his definition of the complex numbers and the operations of the quaternions, one of the main elements of the vector space could be defined i.e. the geometric vectors, where later on \cite{cayley1858ii}. As is known, the theory of matrices was postulated from systems of linear equations, the definition of the matrix and its properties, within these definitions, implicitly, provides an example of vector space. In \cite{grassmann1844lineale} different crucial concepts are developed in vector spaces, such as linear dependence, linear independence, among others, also defining the concept of vector in terms of vector spaces without formalizing the latter, so that in \cite{peano1887integrazione} the formal definition of vector space is presented in an axiomatic way.

There are several extensions of vector spaces in fuzzy logic called fuzzy vector spaces \cite{colloc2020fvsoomm,colloc2017epice,katsaras1977fuzzy}. Regarding its applications we can find, for example, application to cognitive dissonance and decision making, Time fuzzy vector space (i.e. TFVS) modeling of emotion in the purchase decision, application to gaming addiction, modeling of time in medical decisions \cite{colloc2020fvsoomm,colloc2017epice}. Regarding some generalizations of ordered vector spaces \cite{cristescu1976ordered} we can find that of ordered semi-vector spaces over a weak semi-field \cite{milfont2021aggregation}, $n$-dimensional fuzzy sets over a non-empty set, among others \cite{milfont2021aggregation, costa2015generalized,vasantha1993semivector, kandasamy2002smarandache,rodriguez2016position}.

Namely, the aggregation functions play an important role in several areas, including fuzzy logic, decision-making, expert systems, risk analysis, and image processing.  The aggregation functions combine input values into a single output value, which represents all the inputs. Typical examples are weighted means, medians, OWA functions, $t$-norms and $t$-conorms. But there are many other aggregation functions and infinitely many members in most families. In general, there are several contributions to the study of aggregation functions, namely { \cite{beliakov2016practical, dimuro2017new,   yang2020some}, which, we can find them in multi-criteria decision making problems \cite{beliakov2016practical, da2022towards}, connectives in fuzzy logic \cite{beliakov2016practical,Reiser13},  image processing \cite{de2021affine, dimuro2020state,  galar2013aggregation, gonzalez2014use}, in IoT \cite{OliveiraADYRB22}, classifier ensemble \cite{Batista22,Costa18}, forest fire detection \cite{da2022towards}, power quality diagnosis \cite{nolasco2019wavelet}, motor-imagery-based brain-computer interface \cite{fumanal2021motor}, in medicine to estimate the risks of a person to develop a disease \cite{saleh2018learning}, the increasing interest in the study of this topic by defining new classes of aggregation functions \cite{de2021affine}. Among some of the classes to be defined we can find, for example, the pre-aggregation, variants of the Choquet integral and OWA functions \cite{dimuro2020state,BustinceMFGPADB21,LuccaSDBB20,LuccaSDBMKB16,zhu2019nested}.}

In this work, we extend the theory of functions of average type. The space studied here is that given by the cartesian product of fuzzy triangular numbers denoted by  $\TFNR^n$, considering as a set of scalars the field of crisp fuzzy numbers denoted by $\crispy{\con{R}}$. We prove that this set fulfills properties such as associativity, commutativity, additive neutral both for the addition of elements of $\TFNR^n$ as by the product of scalars with elements of $\TFNR^n$. We introduce the concept of increasing average type function proving properties and showing some examples. To study these increasing functions we have defined a linear order on the fuzzy triangular numbers by studying arithmetic properties related to the order. 

This paper is organized as follows: In Section \ref{preliminares}, in addition to establishing the notation used, we recall some essential notions for the remaining sections. In Section \ref{totalorder} study of arithmetic properties and order. We introduce the ordered vector spaces on $\TFN$ in Section \ref{TriangularfuzzySpace}. Section \ref{Averaging-Functions}, we define aggregation functions of the average type on $\TFN$. In addition, we give examples and a characterization of it. 

\section{Preliminaries}\label{preliminares}

In this section, we provide the fuzzy number context and the main {concepts and results which are necessaries in  the remain of this paper}. Let $\con{R}^3 = \con{R}\times\con{R}\times\con{R}$, with $\con{R}$ the set of real numbers, $\con{Q}$ and $\con{I}$ denote the set of rational and irrational numbers, respectively. The notation $<$, $>$, $\leq$ or $\geq$ stands for usual order of $\con{R}$, and $[a,b]$, $]a,b]$, $[a,b[$ and $]a,b[$ the intervals of $\con{R}$ closed, right-closed, right-open,  and open respectively (see \cite{apostol2007calculus,
rudin1964principles}). 

{
\begin{definicion} \cite{klir1995}
 A fuzzy set $A$ over a reference non empty set $X$ is a function $A:X\rightarrow [0,1]$.
\end{definicion}
}

%\vspace{5mm}
\begin{definicion}\cite{klir1995} \label{def-FN} A fuzzy set $A$ over $\con{R}$ is called a fuzzy number if it satisfies the following conditions
\begin{enumerate}
\item[(i)] $A$ is normal, i.e., $\Sup{x\in\con{R}} A(x)=1$;
\item[(ii)] $\cortea{A}$ is a closed interval for every $\alpha\in ~]0, 1]$;
\item[(iii)] $\supp A$, i.e., the support of $A$, is bounded,
\end{enumerate}
where $\nicefrac{A}{\alpha}$ is the $\alpha$-level (or $\alpha$-cut) of $A$.
\end{definicion}

A fuzzy number $A$ is a crisp fuzzy number or in short crisp if there exists $r\in\mathds{R}$ such that $A(r)=1$ and $A(x)=0$ for each $x\neq r$. In this case we will denote $A$ by $\crispy{r}$. Finally, $\NFR$ will denote the set of all fuzzy numbers and $\crispy{\con{R}}$, $\crispy{\con{Q}}$ and $\crispy{\con{I}}$ will denote the set for all crisp fuzzy number{s} with $r\in\con{R}$, $r\in\con{Q}$ and  $r\in\con{I}${, respectively}. \vspace{5mm}

\begin{proposicion}\cite{dubois1978h} Let $A,B\in \mathcal{F}(\con{R})$ then the fuzzy sets $A+B$, $A-B$, $A\cdot B$ and $A\div B$ defined {for each $x\in \con{R}$} by
\begin{enumerate}
\item $(A+B)(x)=\Sup{x=y+z} \min\{A(y),B(z)\}$,
\item $(A-B)(x)=\Sup{x=y-z} \min\{A(y),B(z)\}$,
\item $(A\cdot B)(x)=\Sup{x=yz} \min\{A(y),B(z)\}$,
\item $(A\div B)(x)=\Sup{x=\frac{y}{z}} \min\{A(y),B(z)\},$ since $0\not\in\supp B$,
\end{enumerate}
 is a fuzzy number.
\end{proposicion}
\vspace{5mm}

%\vspace{5mm}
\begin{definicion}\cite{klir1995} \label{deftriangularnumero}
A fuzzy number $A$, is called a triangular fuzzy number, in short $\TFN$, if there is $(a,b,c)\in \mathds{R}^3$ such that $a\leq b\leq c$ and
\begin{equation*}
A(x)=\left\{\begin{array}{ll}
1, & \textrm{ if }x=b,\\
\frac{x-a}{b-a}, & \textrm{ if } x \in~ ]a,b[,\\
\frac{c-x}{c-b}, & \textrm{ if }x \in~]b,c[,\\
0,& \textrm{ other cases. }
\end{array}\right.
\end{equation*}
\end{definicion}
%\vspace{5mm}
In this work, we represent the fuzzy triangular numbers by the triple $(a,b,c)$ and denote the set of all triangular fuzzy numbers by  $\TFNR$, namely $$\TFNR=\{A\in\NFR:A\textrm{ is a triangular fuzzy numbers}\}.$$ %(-a,0,a)
{Observe that each triple $(a,b,c)\in \mathds{R}^3$ satisfying $a\leq b\leq c$ denotes a $\TFN$.}
\begin{proposicion}\cite{dubois1978h}\label{pro-oper}%\cite{Klir1995}
If $A=(a_1,b_1,c_1)$ and $B=(a_2,b_2,c_2)$ be are elements in $\TFNR$ and $\crispy{r}\in\crispy{\con{R}}$.  Then
\begin{enumerate}
\item $A+B=(a_1+a_2,b_1+b_2,c_1+c_2)$,
\item $A-B=(a_1-c_2,b_1-b_2,c_1-a_2)$,
\item $\crispy{r}A=\begin{cases}(ra_1,rb_1,rc_1)\textrm{ , if }r\geq 0,\\
(rc_1,rb_1,ra_1) \textrm{ , if }r<0.\end{cases}$
\end{enumerate}
i.e. the set {$\TFNR$ is} closed under addition, substraction and scalar product. Note that $(-1)A=-A=-(a_1,a_2,a_3)=(-a_3,-a_2,-a_1)$.
\end{proposicion}
\begin{corolario}
Let $A$ be a fuzzy {number}. Then $\underset{n-\textrm{times}}{\underbrace{A+A+...+A}}=\crispy{n}A$.
\end{corolario}
\begin{observacion}
The product and division, i.e.,  $A\cdot B$ and $A\div B$,  of triangular fuzzy numbers only are $\TFN$'s if and only if $A$ or $B$ is {a crisp} fuzzy number (obviously $0\not\in\supp{B}$ in the case of division).
\end{observacion}

 \begin{observacion}\label{rem-ult-sec2}
The addition is commutative and associative in fuzzy triangular numbers. Also, addition distributes over the product and vice versa, where the product is given in Proposition \ref{pro-oper}. Note also that the crisp fuzzy numbers with the operations of addition and multiplication are isomorphic to the field of real numbers and are therefore also a field.
\end{observacion}

\begin{definicion}\label{definicionkytfn} Let $A=(a_1,b_1,c_1)$ and $B=(a_2,b_2,c_2)$ be two $\TFN$'s. We define the relation $\leq_N$ by
$$A\leq_N B \textrm{ if and only if } a_1\leq a_2\textrm{ and }b_1\leq b_2 \textrm{ and } c_1\leq c_2.$$
We denote $A<_N B$ when $A\neq B$, i.e. $a_1\neq a_2$ or $b_1\neq b_2$ or $c_1\neq c_2$.
\end{definicion}

The Definition \ref{definicionkytfn} generalizes the usual order of real numbers, in the sense that both agree when to restrict to crisp fuzzy numbers. It is worth noting that this relation is the restriction to the $\TFN$'s of the Klir and Yuan partial order on the fuzzy numbers. In the following, we adapt the definition of admissible orders on fuzzy numbers introduced in \cite{zumelzu2022} to the context of $\TFN$'s.

\begin{definicion}\label{def-ordenadmisible}
Let $\leq_N$ and $\preceq$ be two orders on the $\TFN$'s. The order $\preceq$ is called an admissible order w.r.t. {$\prt{{\TFNR},\leq_N}$}, if
\begin{enumerate}[labelindent=\parindent, leftmargin=*,label=]
\item[(i)] $\preceq$ is a linear order on ${\TFNR}$;
\item[(ii)] for all $A$, $B$ in ${\TFNR}$, $A\preceq B$ whenever $A\leq_N B$.
\end{enumerate}
\end{definicion}

\section{Admissible order}\label{totalorder}
In this section, we focus on proving {an admissible order on $\TFNR$,} we also study properties that such linear {order} verifies relates to the addition and subtraction operations.
 \vspace{5mm}
 
\begin{definicion}\label{defOT}
  Let $(a_1,b_1,c_1)$ and $(a_2,b_2,c_2)$ be two $\TFN$'s. We define the relation $\leqot$ by
\begin{multline*}
(a_1,b_1,c_1)\leqot (a_2,b_2,c_2) \textrm{ if and only if }  \begin{cases} b_1<b_2,
\textrm{ or }\\
 b_1=b_2\textrm{ and }c_1<c_2, \textrm{ or }\\
 b_1=b_2\textrm{ and }c_1=c_2\textrm{ and }a_1\leq a_2.\end{cases}
\end{multline*}
We denote $A<_{\OT} B$ when {$A\leqot B$ and $A\neq B$, i.e.} $a_1\neq a_2$ or $b_1\neq b_2$ or $c_1\neq c_2$.
\end{definicion}

\begin{proposicion}\label{OTordentotal}
The relation $\leqot$ is an admissible order on  $\TFNR$.
\end{proposicion}

\begin{proof}
Let $A=(a_1,b_1,c_1)$, $B=(a_2,b_2,c_2)$, $C=(a_3,b_3,c_3)\in\TFNR$. 

\ITEM{1.} Reflexivity: Straightforward from Definition \ref{defOT}.

\ITEM{2.} Antisymmetry: Let $A\leqot B$ and $B\leqot A$. If $b_1\neq b_2$ then either $A<_{OT}$ B or $B<_{OT}$ A but not both. If $b_1=b_2$ and $c_1\neq c_2$ then, $A<_{OT} B$ or $B<_{OT} A$ but not both. If $b_1=b_2$ and $c_1=c_2$ and $a_1\neq a_2$ then either $A<_{OT} B$ or $B<_{OT} A$ but not both. Thereby, $a_1=a_2$, $b_1=b_2$ and $c_1=c_2$.

\ITEM{3.} {Transitivity:} If $A\leqot B$ and $B\leqot C$ then
\begin{enumerate}
\item $b_1<b_2$ or ($b_1=b_2$ and $c_1<c_2$) or ($b_1=b_2$ and $c_1=c_2$ and $a_1\leq a_2$),
\item $b_2<b_3$ or ($b_2=b_3$ and $c_2<c_3$) or ($b_2=b_3$ and $c_2=c_3$ and $a_2\leq a_3$).
\end{enumerate}
This proof consists of 9 steps, although here we will demonstrate 4.  Step 1: If $b_1<b_2$ then $b_1<b_3$ and therefore $(a_1,b_1,c_1)<_{\OT}(a_3,b_3,c_3)$. { Step 2: If $b_1<b_2$ and $b_2=b_3$ and $c_2<c_3$ then $b_1<b_2=b_3\leq c_2<c_3$, and therefore $(a_1,b_1,c_1){\leqot} (a_3,b_3,c_3)$. Step 3: If $b_1<b_2$ and $b_2=b_3$ and $c_2=c_3$ and $a_2\leq a_3$ then $b_1<b_3$. Step 4: If $b_1=b_2$ and $c_1=c_2$ and $a_1\leq a_2$ and $b_2=b_3$ and $c_2<c_3$ then $a_1\leq a_2\leq b_1=b_2=b_3\leq c_1=c_2<c_3$. Hence, $c_1<c_3$}. Therefore $(a_1,b_1,c_1)\leq_{OT}(a_3,b_3,c_3)$. 

\ITEM{4.} Totality: Let $(a_1,b_1,c_1)\neq (a_2,b_2,c_2)$. Then $a_1\neq a_2$ or $b_1\neq b_2$ or $c_1\neq c_2$. So, by Definition \ref{defOT} and trichotomy property of real numbers where for any two real numbers, $x$ and $y$, $x<y$ or $y> x$ or $x=y$, it is $(a_1,b_1,c_1)<_{OT}(a_2,b_2,c_2)$ or $(a_2,b_2,c_2)<_{OT}(a_1,b_1,c_1)$.

\ITEM{5.} {Refinement}: Let $(a_1,b_1,c_1)\leq_N (a_2,b_2,c_2)$. If $A=B$ then, since $\leqot$ is reflexive, we have that $(a_1,b_1,c_1)\leqot (a_2,b_2,c_2)$.
If $(a_1,b_1,c_1)\leq_N (a_2,b_2,c_2)$ then by Definition \ref{definicionkytfn} we have $a_1\leq_N a_2$ and $b_1\leq_N b_2$ and $c_1\leq_N c_2$ and ($a_1\neq a_2$ or $b_1\neq b_2$ or $c_1\neq c_2$). Then, if $b_1\neq b_2$ hence $b_1 <b_2$, but, if $b_1=b_2$ and $c_1\neq c_2$ then $b_1=b_2$ and $c_1< c_2$. Finally, if $b_1=b_2$ and $c_1= c_2$ and $a_1\neq a_2$ then $b_1=b_2$ and $c_1= c_2$ and $a_1< a_2$.  Thereby, $A\leqot B$.

In this way the Proposition is proven. 
\end{proof}

\begin{definicion}
Let $A\in \TFNR$. Then,
\begin{enumerate}
\item $A$ is $OT$-positive if $A >_{OT}\crispy{0}$.
\item $A$ is $OT$-negative if $A <_{OT}\crispy{0}$.
\end{enumerate}
\end{definicion}

In $\TFNR$ the elements of the form $(-a,0,a)$ with $a\leq 0$ we calling $0$-\textit{isosceles triangular {fuzzy numbers}} and  {denoted by $\TFNR_0$ to the set of all $0$-isosceles $\TFN$'s. In} particular $\crispy{0}=(0,0,0)$ belongs to $\TFNR_0$.
\begin{tricotomia}\label{positovoornegativo}
If $A\in \TFNR$ then one and only one of the following statements is true: 
\begin{enumerate}
\item $A$ is $OT$-positive;
\item $A$ is $OT$-negative;
\item $A=\crispy{0}$.
\end{enumerate}
\end{tricotomia}

\begin{proof} Straightforwardly, because
$\leqot$ is a linear order on
$\TFN$.
\end{proof}

{Observe that the trichotomy law non holds case consider the natural order  instead of $\leq_{OT}$.}

\begin{proposicion}\label{proposiciondeB}
Let $(a,b,c)$ be a $\TFN$. Then, $(a,b,c)$ is $OT$-positive if and only if $b>0$ or $(0=b< c)$. In the same way, $(a,b,c)$ is $OT$-negative if and only if $b< 0$ or $a<b=c=0$.
\end{proposicion}
\begin{proof} Straightforwardly from above definition.
\end{proof}
\begin{corolario}\label{positivonegativo}
Let $A\in\TFNR$. If $A$ is $OT$-negative then $-A$ is $OT$-positive. We note that the converse in general is not true.
\end{corolario}

\begin{ejemplo}
Let $(-1,0,1)$ is $OT$-positive. Note that $-(-1,0,1)=(-1,0,1)$ is also $OT$-positive.
\end{ejemplo}

\begin{corolario}\label{corolarioceroespecial}
Let $A\in\TFNR_0$. If $A\neq\crispy{0}$ then ${-A}$ is $OT$-positive.
\end{corolario}

\begin{proposicion}\label{proposicionrestadacasicero}
If $A\in \TFNR$ then $A-A\in \TFNR_0$.
\end{proposicion}
\begin{proof}
Let $A=(a,b,c)\in{\TFNR}$. Then $(a,b,c)-(a,b,c)=(a-c,0,c-a)$ i.e. $c-a=-a+c=-(a-c)$ but $a<c$ hence $c-a>0$. In this way the Proposition is prove.
\end{proof}
In the following corollary, we obtain a property of over-additive inverse.
\begin{corolario}\label{teoinverso}
Let $A$ be a ${\TFN}$. Then $A-A\geq_{OT}\crispy{0}$. In addition, $A-A=\crispy{0}$ if and only if $A$ is crisp.
\end{corolario}

\begin{proposicion} 
 Let $A$, $B$ two $\TFNR$. $A\leqot B$ if and only if $B-A\geqot\crispy{0}$.
\end{proposicion}

\begin{proof}
Let $\NFT{1}$, $\NFT{2}\in\TFNR$ such that $\NFT{1}\leqot\NFT{2}$. If $b_1<b_2$ then $0<b_2-b_1$. If $b_1=b_2$ and $a_1\leq c_1<c_2$ hence $0<c_2-a_1$. If $b_1=b_2$, $c_1=c_2$ and $a_1\leq a_2$ then {$c_2-a_1\geq 0$. Observe that case $c_2-a_1= 0$ then $a_1=b_1=c_1=a_2=b_2=c_2$}. Therefore, if $A\leqot B$ then $B-A\geqot\crispy{0}$. It is not hard to see that the converse is true from Definitions \ref{deftriangularnumero} and \ref{defOT}.
\end{proof}
\begin{observacion}
The order {$\preceq$} proposed in \cite{valentina2021} does not satisfy the previous property {since $(1,2,3)\prec (1,2,4)$, and  $(1,2,4)-(1,2,3)=(-2,0,3)\prec\crispy{0}$}.
\end{observacion}
\vspace{5mm}

\begin{teorema}
Let $A$ and $B$ in $\TFNR$. If $A$ and $B$ are $OT$-positive then $A+B$ is also $OT$-positive.
\end{teorema}

\begin{proof}[Proof.] Let $A=(a_1,b_1,c_1)$ and $B=(a_2,b_2,c_2)$ two $OT$-positive $\TFNR$. Then, from Proposition \ref{proposiciondeB} we have that $b_1\geq 0$, $b_2\geq 0$, $c_1>0$ and $c_2>0$. So, (1) if $b_1>0$ or $b_2>0$ then $b_1+b_2>0$. (2) if $b_1=b_2=0$ then $b_1+b_2=0$ and, because $c_1>0$ and $c_2>0$ and $c_1+c_2>0$. Therefore, in both cases, $A+B>_{OT}\crispy{0}$.
\end{proof}

\begin{teorema}
Let $A$, $B$ and $C$ in $\TFNR$. If $A\leqot B$ then $A+C\leqot B+C$.
\end{teorema}

\begin{proof}[Proof.] Let $A=(a_1,b_1,c_1)$, $B=(a_2,b_2,c_2)$ and $C=(a_3,b_3,c_3)$ be three $\TFN$'s. If  $A\leqot B$ we get the following cases
\begin{enumerate}
\item $b_1<b_2$ or
\item $b_1=b_2$ and $c_1<c_2$ or
\item $b_1=b_2$ and $c_1=c_2$ and $a_1\leq a_2$.
\end{enumerate}
For the first case, we have: If $b_1<b_2$ then $b_1+b_3<b_2+b_3$. Hence $A+C\leqot B+C$. The other cases are analogous.

In this way the Proposition is prove.
\end{proof}
\begin{corolario}\label{sumacontinua}
Let $A$, $B$ $C$ and $D$ be $\TFN$. If $A\leqot B$ and $C\leqot D$ then $A+C\leqot B+D$.
\end{corolario}

{
\begin{proposicion}
Let $A$ and $B$ be two $\TFN$ and $\crispy{r}\in\crispy{\con{R}}$. \begin{enumerate}
\item If $A\leqot B$ and $\crispy{r}\geqot \crispy{0}$ then $\crispy{r}A\leqot \crispy{r}B$,
\item If $A\leqot B$ and $\crispy{r}\leqot \crispy{0}$ then $\crispy{r}A\geqot \crispy{r}B$.
\end{enumerate} 
\end{proposicion}
\begin{proof}
 Straightforward.
\end{proof}
}

\begin{definicion}
{Given $A$, $B$ in $\TFNR$ we} define:
\begin{enumerate}
\item Maximum: \begin{center}
$\Maxx_{\leqot }\{A,B\}=\begin{cases} A,\textrm{ if }B\leqot A;\\
\textrm{ or }\\
B,\textrm{ if }A\leqot B.\end{cases}$
\end{center}
\item Minimum: \begin{center}
$\Minn_{\leqot}\{A,B\}=\begin{cases} B,\textrm{ if }B\leqot A;\\
\textrm{ or }\\
A,\textrm{ if }A\leqot B.\end{cases}
$\end{center}
\end{enumerate}
\end{definicion}

\begin{observacion}
The algebras $(\NFR,+)$ and $(\NFR,\cdot)$,  are not invertible commutative monoids (see \cite[pag.7]{kandasamy2002smarandache}). On the other hand $(\NFR,+,\cdot)$ is not a semiring in the sense \cite{kandasamy2002smarandache}.
\end{observacion}

\section{Triangular Fuzzy Space}\label{TriangularfuzzySpace}
An $n$-tuple of triangular fuzzy numbers $(A_{1},A_{2},A_{3},\ldots,A_{n})$ for an integer $n\geq 1$ is called an $n$-dimensional point or an $n$-dimensional vector. In this case, the $\TFN$'s $A_{1},A_{2},A_{3},\ldots,A_{n}$ are called of coordinates or components of the vector. The collection of all $n$-dimensional vectors is called the vector space of $n$-tuples, or simply $n$-space. We denote this space by $\TFNR^n$. In this paper we shall usually denote vectors by  $\Vector{A},\Vector{B},\Vector{C},\ldots$ and components by the corresponding letters $A$, $B$, $C,~\ldots$ with a subindex in $\{1,...,n\}$. Thus, we write $$\Vector{A}=(A_{1},A_{2},A_{3},\ldots,A_{n}).$$ To convert $\TFNR^n$, into an algebraic system, we introduce two vector operations called addition and multiplication by scalars. The word scalar is used here as a synonym for crisp fuzzy number.

\begin{definicion}\label{operacionesvectoriales}
Two vectors $\Vector{A}$ and $\Vector{B}$ in $\TFNR^n$. The sum $\Vector{A}+\Vector{B}$ is defined to be the vector obtained by adding corresponding components: $$\Vector{A}+\Vector{B}=(A_{1}+B_{1},A_{2}+B_{2},\ldots,A_{n}+B_{n}).$$ If $\crispy{r}\in\crispy{\con{R}}$, we define $\crispy{r}\Vector{A}$ or $\Vector{A}\crispy{r}$ to be the vector obtained by multiplying each component
of $\Vector{A}$ by $\crispy{r}$: $$\crispy{r}\Vector{A}=(\crispy{r}A_{1},\crispy{r}A_{2},\crispy{r}A_{3},\ldots,\crispy{r}A_{n})$$
\end{definicion}
From this definition it is easy to verify the following properties of these operations

\begin{proposicion}\label{operacionesenspacevector}
Vector addition is commutative,
$$\Vector{A}+\Vector{B}=\Vector{B}+\Vector{A},$$
and associative
$$\Vector{A}+(\Vector{B}+\Vector{C})=(\Vector{A}+\Vector{B})+\Vector{C}.$$
Multiplication by scalars is associative,
$\crispy{r}({\crispy{t}\Vector{A}}) = (\crispy{r}\crispy{t})\Vector{A}$
and satisfies the two distributive laws
$$\crispy{r}(\Vector{A}+\Vector{B})={\crispy{r}\Vector{A}+\crispy{r}\Vector{B},}$$ and $$(\crispy{r} + \crispy{t})\Vector{A} ={ \crispy{r}\Vector{A} + \crispy{t}\Vector{A}.}$$
\end{proposicion}

\begin{proof} We will only prove the commutative and distributive law. Let $\Vector{A}$ and $\Vector{B}$ two $\TFNR^n$ and $\crispy{r}\in\crispy{\con{R}}$. For the commutativity we have
\begin{align*}
\Vector{A}+\Vector{B} %= & (A_1,A_2,...,A_n)+(B_1,B_2,...,B_n)\\
 = & (A_1+B_1,A_2+B_2,...,A_n+B_n)\\
  = & (B_1+A_1,B_2+A_2,...,B_n+A_n)\\
 % = & (B_1,B_2,...,B_n)+ (A_1,A_2,...,A_n)\\
    = & \Vector{B}+\Vector{A}.
\end{align*}
On the other hand, for the distributive law we obtain
\begin{align*}
(\crispy{r}+\crispy{t})\Vector{A} = & ((\crispy{r}+\crispy{t})A_1,(\crispy{r}+\crispy{t})A_2,...,(\crispy{r}+\crispy{t})A_n)\\
 = & (\crispy{r}A_1+\crispy{t}A_1,\crispy{r}A_2+\crispy{t}A_2,...,\crispy{r}A_n+\crispy{t}A_n) \,\,\mbox{{ by Remark \ref{rem-ult-sec2}}}\\
= & (\crispy{r}A_1,\crispy{r}A_2,...,\crispy{r}A_n)+(\crispy{t}A_1,\crispy{t}A_2,...,\crispy{t}A_n)\\
%  = & (B_1+A_1,B_2+A_2,...,B_n+A_n)\\
  = & \crispy{r}(A_1,A_2,...,A_n)+ \crispy{t}(A_1,A_2,...,A_n)\\
    = & \crispy{r}\Vector{A}+\crispy{t}\Vector{A}.
\end{align*}
The proof of the remaining properties, we use similar arguments to those given above. Concluding our proof.
\end{proof}
%\rot{The $gH$-difference for fuzzy numbers can be defined as follows:} 
%\begin{definicion} Given $u,v \in \NFR$, the $gH$-difference is the fuzzy number $w$, if it exists, such that $$u\circleddash_{g}v=w \Leftrightarrow \left\{ \begin{array}{lcc} i)\, u=v+w \\ or\, ii)\, v=u+(-1)w \\ \end{array} \right.$$ \end{definicion}

\begin{observacion}\label{observaciondeorden}
The vector with all components $\crispy{0}$ is called the zero vector and is denoted by $\Vector{0}$. It has the property that $\Vector{A}+\Vector{0} = \Vector{A}$ for every vector $\Vector{A}$; in other words, $\Vector{0}$ is an identity element for vector addition. The vector $(-\crispy{1})\Vector{A}$ is also denoted by $-\Vector{A}$ and is called the opposite of $\Vector{A}$ and notice that the equation $\Vector{A}+(-\Vector{A}) = \Vector{0}$ is generally not true. For this, it suffices to consider $\Vector{A}\in \TFNR^n$ wet get
$\Vector{A}-\Vector{A}=(A_1-A_1,...,A_n-A_n)$ from Corollary \ref{teoinverso} have $A_i-A_i>_{OT} \crispy{0}$ whenever $A_i$ is not crisp. Note also that $\Vector{A}\crispy{0} =\Vector{0}$ and that $\Vector{A}\crispy{1}=\Vector{A}$ for all $\Vector{A}\in\TFNR^n$. Therefore, $\TFNR^n$, endowed with the addition and scalar product of Definition \ref{operacionesvectoriales}, is not a vector space on the field of crisp fuzzy numbers. Nevertheless, it is a semi-vector space in the sense of Vasantha Kandasamy \cite{kandasamy2002smarandache} and \cite{milfont2021aggregation}.
\end{observacion}

\begin{definicion}
Let $(A_1,A_2,...,A_n)$ and $(B_1,B_2,...,B_n)$ in $\TFNR^n$. We define the relation $\leq_{\OT^n}$ by
$$(A_1,A_2,...,A_n)\leq_{\OT^n}(B_1,B_2,...,B_n)\Longleftrightarrow \left\{\begin{array}{c}A_1\leqot B_1\textrm{ and }\\
A_2\leqot B_2\textrm{ and }\\ \vdots \\A_n\leqot B_n.\end{array}\right.$$
\end{definicion}

\begin{proposicion}
Let $\Vector{A}\in\TFNR^n$. Then $\Vector{A}-\Vector{A}\geq_{OT^n}\Vector{0}$.
\end{proposicion}
\begin{proof}
Straightforward from Remark \ref{observaciondeorden}.
\end{proof}

The set $\TFNR^n_0$ denote the set of elements of the form $(A_1,...,A_n)$ such that $A_i\in\TFNR_0$ for all $i=1,...n$.

\begin{proposicion}
If $\Vector{A}\in\TFNR^n$ then $\Vector{A}-\Vector{A}\in \TFNR^n_0$.
\end{proposicion}
\begin{proof}
Straightforward from Proposition \ref{proposicionrestadacasicero} and Remark \ref{observaciondeorden}.
\end{proof}

\section{Averaging Functions on $\TFN$}\label{Averaging-Functions}

  \begin{definicion}
A function $\operador{E}{\TFNR^n}{\TFNR}$ is $\OT$-increasing if 
 $$E(\Vector{A})\leqot E(\Vector{B})$$ whenever $\Vector{A}\leq_{\OT^n} \Vector{B}$. 
We can analogously define the $\OT$-decreasing function.
 \end{definicion}
 
Let a function $\operador{E}{\TFNR^n}{\TFNR}$. We call $E$ function type average (in short $\OT$-FTA) if it is $\OT$-increascing and, for each $A_1$, ..., $A_n\in\TFNR$,
$$
\Minn_{\leqot}\{A_1,...,A_n\} \leqot E(A_1,...,A_n) \leqot \Maxx_{\leqot} \{A_1,...,A_n\}.$$

 \begin{definicion}\label{idempotency}
 A function $\operador{E}{\TFNR^n}{\TFNR}$ is idempotent if $E(A,...,A)=A$ for every  $A$ in $\TFNR$.
 \end{definicion}
 
 \begin{proposicion}\label{typeaverange}
 Let $E$ be an $\operador{E}{\TFNR^n}{\TFNR}$. If $E$ is $\OT$-increasing then the following statements are equivalent:
 \begin{enumerate}%[labelindent=\parindent, leftmargin=*,label=]
 \item[(i)] $E$ is idempontent;
 \item[(ii)] $E$ is {an OT}-FTA.
 \end{enumerate}
 \end{proposicion}
 \begin{proof} $(i)\Rightarrow (ii)$ Take any $A_1$, ..., $A_n$ in $\TFNR$. We denote by $A_i=\Minn_{\leqot}\{A_k{;}k=1,...,n\}$ and $A_j=\Maxx_{\leqot }\{A_k{;}k=1,...,n\}$. By monotonicity 
$$A_i=E(A_i,...,A_i)\leqot E(A_1,...,A_n) \leqot E(A_j,...,A_j)=A_j.$$
Hence 
$$\Minn_{\leqot}\{A_k;k=1,...,n\}\leqot E(A_1,...,A_n) \leqot \Maxx_{\leqot }\{A_k;k=1,...,n\}.$$
$(ii)\Rightarrow (i)$ Straightforward.
 \end{proof}
 \begin{definicion}
Let $A_i$ on $\TFNR$ for all $i=1,...,n$. 
\begin{enumerate}%[labelindent=\parindent, leftmargin=*,label=]
\item[(i)] Arithmetic mean
$$M_a(A_1,...,A_n)={\left( \suma{n}{i=1}A_i \right ) \div\crispy{n}}.$$

\item[(ii)] Weighted arithmetic mean
$$M_w(A_1,...,A_n)=\suma{n}{i=1}\crispy{a}_iA_i,$$ where $\crispy{a}_i\geqot\crispy{0}$ for all $i=1,...,n$ and $\suma{n}{i=1}\crispy{a}_i=\crispy{1}$.
\end{enumerate}
 \end{definicion}

  \begin{proposicion}
 The function $\operador{M_w}{\TFNR^n}{\TFNR}$ is {an} $\OT$-FTA.
 \end{proposicion}
 \begin{proof} Let $A_i=\Minn_{\leqot}\{A_k{;}k=1,...,n\}$ and $A_j=\Maxx_{\leqot }\{A_k{;}k=1,...,n\}$ be two $\TFNR$. By monotonicity 
$$A_i=A_i\suma{n}{i=1}\crispy{a}_i\leqot M_w(A_1,...,A_n) \leqot M_w(A_j,...,A_j)=A_j\suma{n}{i=1}\crispy{a}_i=A_j.$$
Hence 
$$\Minn_{\leqot}\{A_k;k=1,...,n\}\leqot M_w(A_1,...,A_n) \leqot \Maxx_{\leqot }\{A_k;k=1,...,n\}.$$
Therefore, by Proposition \ref{typeaverange} we get $M_w$ is {an} $\OT$-FTA.
 \end{proof} 
\begin{corolario}
 The function $\operador{M_a}{\TFNR^n}{\TFNR}$ is {an $\OT$}-FTA.
\end{corolario}
 
%
%\bibliography{bibliografiads}

%\label{bibliografia}

\end{document}